\documentclass[12pt,a4paper]{article}
\usepackage{authblk}
\usepackage{amsmath, amssymb, amsfonts, amsthm, array}
\usepackage{graphicx, caption, subcaption, float}

\newtheorem{theorem}{Theorem}[section]

\newtheorem{coro}[theorem]{Corollary}

\theoremstyle{definition}

\newtheorem{example}[theorem]{Example}

\title{ \textbf{The characteristic polynomial of a graph containing loops}}

\author[1]{Deepa Sinha\thanks{deepasinha@sau.ac.in}}
\author[1]{Bableen Kaur\thanks{bableenkaur@students.sau.ac.in}}
\author[2]{Thomas Zaslavsky\thanks{zaslav@math.binghamton.edu}}

\affil[1]{Department of Mathematics, South Asian University, Akbar Bhawan, Chanakyapuri,  New Delhi 110021, India}
\affil[2]{Department of Mathematical Sciences, Binghamton University (SUNY), Binghamton, New York, U.S.A. 13902-6000}

\date{\today}

\begin{document}
	\maketitle
	
	\begin{center}{\textbf{Abstract}}\end{center}
	
	\noindent In this article, we focus on the characteristic polynomial of a graph containing loops, but without multiple edges. We present a relationship between the characteristic polynomial of a graph with loops and the graph obtained by removing all the loops. In turn, we compute the characteristic polynomial of unitary addition Cayley graphs.
	
	\noindent \textbf{$\textbf{2010}$ Mathematics Subject Classifications}: 05C50.
	
	
	\newpage
	
	\section{Introduction}
	
	Let $G$ be a graph with vertex set $\{v_{1}, v_{2}, \dots, v_{p}\}$. The characteristic polynomial $\phi(G;x)$(or simply $\phi(G)$) of a graph $G$ is the characteristic polynomial of its adjacency matrix. The adjacency matrix $A(G)=[a_{ij}]$ of a graph $G$ is a square matrix of order $p\times p$ such that the $(i,j)$ entry of $A(G)$ is $1$ if vertices $v_{i}$ and $v_{j}$ are adjacent and $0$ otherwise. The spectrum of an adjacency matrix is a list of its eigenvalues along with their multiplicities. The spectrum of a graph is a spectrum of its adjacency matrix.
	
	\noindent The spectrum and the characteristic polynomial of a graph frequently appear in  mathematical sciences, chemistry, and physics. As to graph theorists, the characteristic polynomial tells information about the structural properties of a graph. On the other hand, in chemistry, the characteristic equation is related to a secular equation formed from the chemical formula of organic molecules, the so-called unsaturated conjugated hydrocarbons. It is useful in predicting the relative stabilities of conjugated hydrocarbons.
	
	\noindent A lone pair is another concept that is informative for a chemist. Lewis introduced it, and it forms the basis for an electronic theory of chemical bonds. The traditional representation of a lone pair by Lewis and Langmuir \cite{lew, lan} involved a pair of dots located near an atom symbol. However, from the topological viewpoint, this is a rather poor image. There is no general convention on how lone pairs may be expressed in classical 2D models, therefore their presence is always ignored. Only in the Gillespie--Nyholm approach to molecular geometry \cite{gil, gilh, giln}, they studied the lone pair (a non bonding domain) as an object equivalent to a bonding domain involving an arrangement of both domain types around an atom.
	
	\noindent Lone pairs describe the formation of donor-acceptor bonds as reflected in the concept of Lewis acids and bases \cite{jen}. Consider the chemical equation, NH$_{3}$ + BH$_{3}$ = NH$_{3}$BH$_{3}$ involving the base ammonia, which has a lone pair, and the acid borane, which has a vacancy (the lack of electron pairs to form the stable octet configuration of a noble gas). The vacancy is hardly represented in molecular graphs on surfaces, although it is related to the depletion of the charge density. Reactions of this sort incorporate the same logical modeling paradox as does the recombination of two free radicals. A molecule with a well-defined graph and well-defined 2D surface is formed from ill-defined model structures. Hence, it is still an open question of how to interpret the lone pair in molecular graphs.
	
	\noindent The precise term molecular graph is ill-defined. For instance, the chemists frequently used these graphs to represent atoms and bonds in a chemical reaction \cite{bal,bon}. One may consider only heavy atoms (as in the so-called hydrogen-suppressed graphs) or the bonds representing only s-frameworks (graphs for p-systems).  A  single vertex may also represent a functional group. These graphs (and even molecular multigraphs) are incomplete in the sense of the original Lewis dot formula that consists of all atoms and all valence electrons (represented by dots). Perhaps the best image of a Lewis formula is the molecular pseudograph, a multigraph with loops representing lone pairs. Only this graph represents all valence electron pairs by edges (including non-bonding lone pairs) and all atoms.
	
	\noindent A clear model consisting of molecular pseudographs appeared in 1973 by Dugundji and Ugi \cite{dugu}. They represent a molecule by a connection table (BE-matrix) that matches the adjacency matrix for a pseudograph with the number of valence electrons for each atom on the main diagonal, which is necessary for a correct count of degree of vertices. The loops appear naturally while reconstructing a graph. Molecular pseudographs appeared in different fields of mathematical chemistry \cite{babh}. However, they are rarely used. Probably one of the reasons is that chemists frequently draw ``lobes" of p-orbitals near the atoms in molecular graphs, and the loops may be confused with p-orbitals. 
	
	\noindent So chemical terms and concepts (that may have an imprecise definition in classical molecular models) need to be translated into the language of pseudographs, because the pseudograph coincides with the Lewis formula. Furthermore, every abstract pseudograph corresponds (if at all) to only a specific finite set of molecular pseudographs. 
	
	\noindent In Section $2$, we discuss how a basic figure containing loops contributes to the coefficients of a characteristic polynomial of a graph. In Section $3$, there are two  main theorems along with other results; in one theorem the characteristic polynomial of a graph with loops is expressed in terms of the characteristic polynomials of its simple subgraphs, and the second theorem helps in computing the characteristic polynomial of a graph without loops from the characteristic polynomial of the same graph containing loops. Subsequently, in Section $4$, we show that the characteristic polynomial of a unitary addition Cayley graph can be computed with the help of an anti-circulant graph.


	\section{Contribution of a loop to the basic figure}

	From the characteristic polynomial of a simple graph, the value of the coefficients was discovered independently by Sachs \cite{sac} and Spialter \cite{spi} as given in the following theorem.  A basic figure $\mathcal{B}$ of a simple graph $G$ is a subgraph of $G$ whose each component is either a cycle or an edge. 
	
	\begin{theorem}\label{thm}
		The characteristic polynomial of a simple graph $G$ is given by the following:  
		\begin{equation*} 
		\phi(G) = \sum_{\mathcal{B} \in \textbf{B}(G)} (-1)^{k(\mathcal{B})} 2^{c(\mathcal{B})}x^{p-|V(\mathcal{B})|},
		\end{equation*}
	where $\textbf{B}(G)$ denotes the set of basic figures of a graph $G$, and $k(\mathcal{B})$ and $c(\mathcal{B})$ denote the number of components and cycles in a basic figure $\mathcal{B}$, respectively.
    \end{theorem}

	\noindent We deal with the case where the graph may contain loops. The diagonal entry $(i,i)$ of an adjacency matrix of this graph is $1$ if there is a loop at the vertex $v_{i}$, and $0$ otherwise. In chemistry, such graphs are used to represent heteroconjugated molecules. In \cite{tri}, the author shows that Sachs's formula (Theorem \ref{thm}) for computing the characteristic polynomial of a simple graph can be extended as it is to compute the characteristic polynomial of a pseudograph (without multiple edges) associated with heteroconjugated molecules. The generalization involves the difference in the set of basic figures only. In this case, a basic figure is a subgraph whose each component is either a cycle (having $3$ or more vertices) or a loop or an edge. Figure \ref{G} and Table \ref{bfig} show an example of a graph $G$ with a loop and its basic figures. A basic figure without a loop of this graph is nothing but a basic figure of a graph without a loop. 
	\begin{figure}[h]
		\centering
		\includegraphics[width=1.2in,height=1.2in]{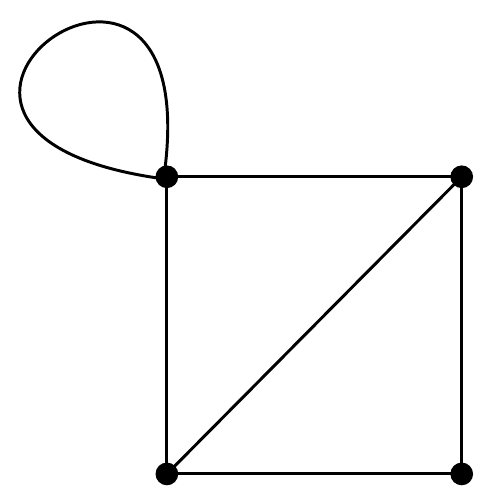}
		\caption{A graph $G$ with a loop.}
		\label{G}
	\end{figure}
    \begin{table}[htp]
    	\centering
    	\begin{tabular}{m{0.5cm}|m{1.2cm}|m{1.2cm}|m{1.2cm}|m{1.2cm}|m{1.2cm}}
    		
    		 \centering\textbf{$n$} & \multicolumn{5}{c}{Basic figures of order $n$} \\ \hline
    		 
    		 \centering\textbf{$1$} & \includegraphics[width=0.5in,height=0.5in]{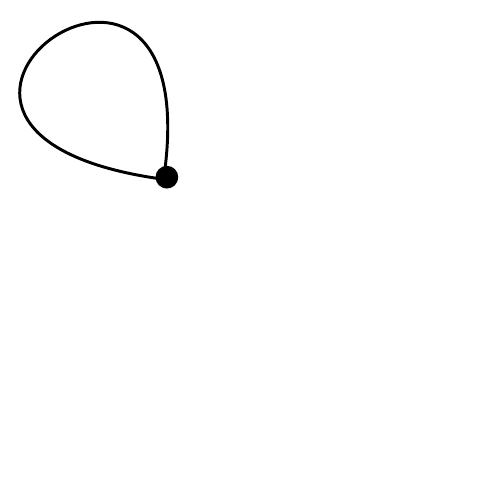} \\ \hline
    		
    		 \centering\textbf{$2$} & \includegraphics[width=0.5in,height=0.5in]{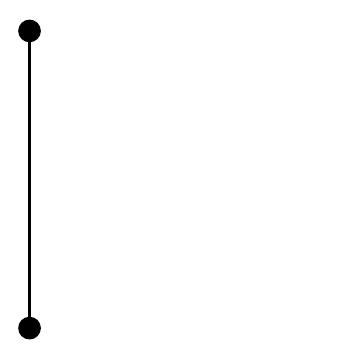} & \includegraphics[width=0.5in,height=0.5in]{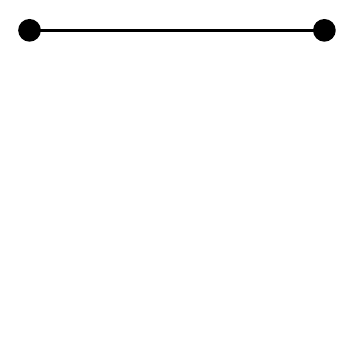} & \includegraphics[width=0.5in,height=0.5in]{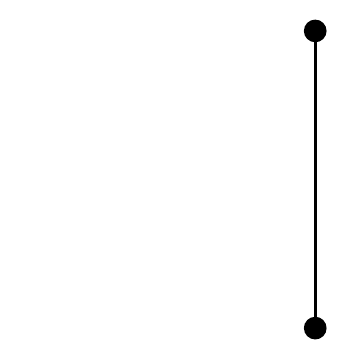} & \includegraphics[width=0.5in,height=0.5in]{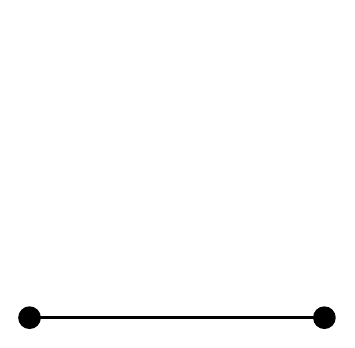} & \includegraphics[width=0.5in,height=0.5in]{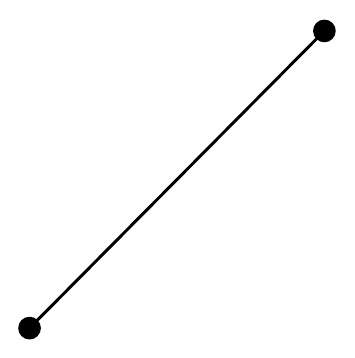} \\ \hline
    	
    		 \centering\textbf{$3$} & \includegraphics[width=0.5in,height=0.5in]{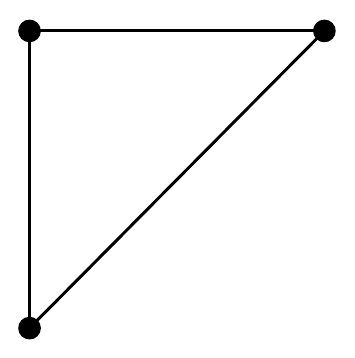} & \includegraphics[width=0.5in,height=0.5in]{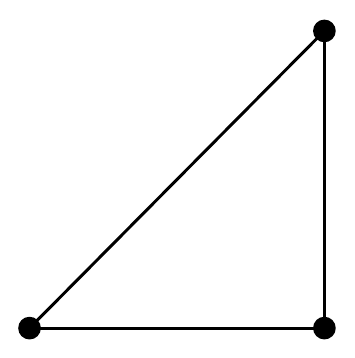} & \includegraphics[width=0.5in,height=0.5in]{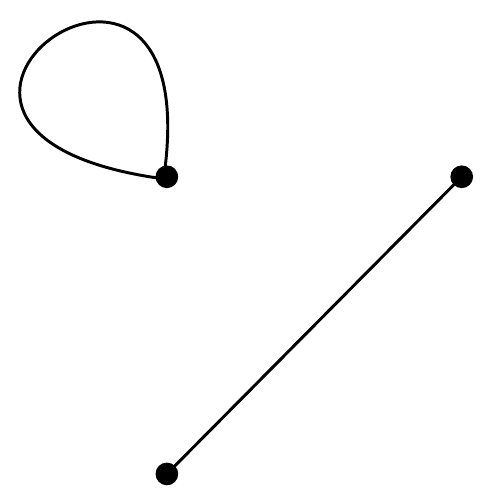} & \includegraphics[width=0.5in,height=0.5in]{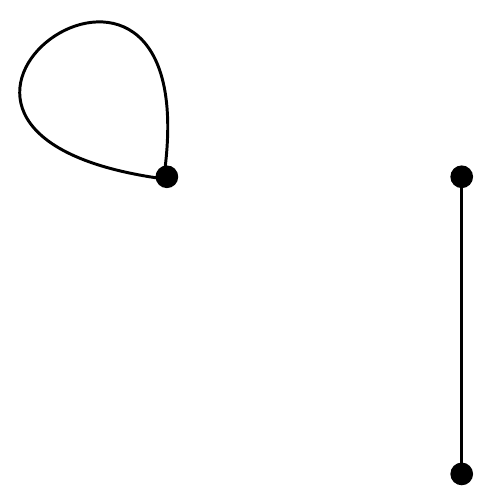} & \includegraphics[width=0.5in,height=0.5in]{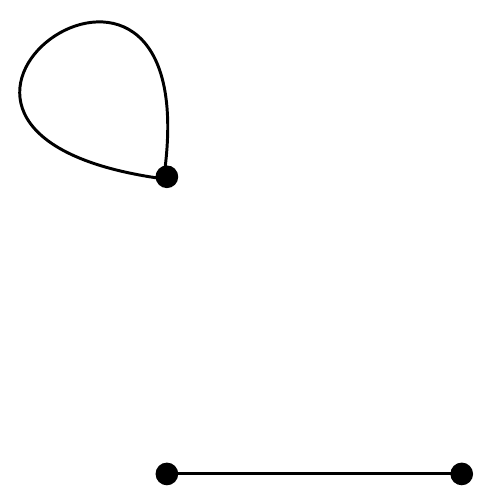} \\ \hline
    		
    		 \centering\textbf{$4$} & \includegraphics[width=0.5in,height=0.5in]{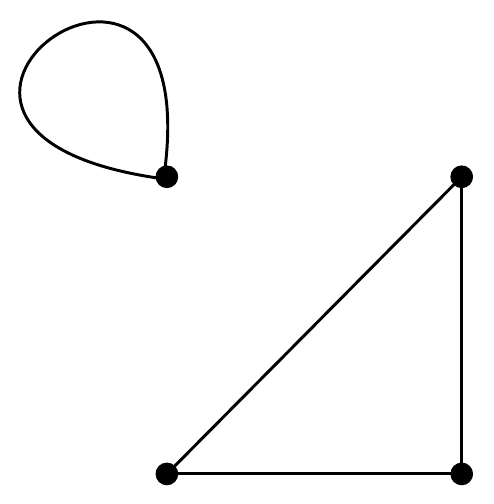} & \includegraphics[width=0.5in,height=0.5in]{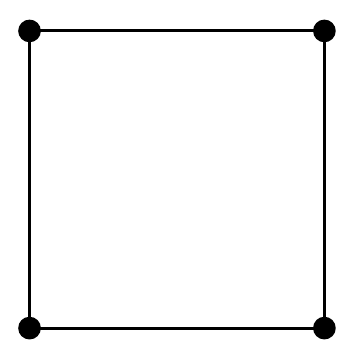} & \includegraphics[width=0.5in,height=0.5in]{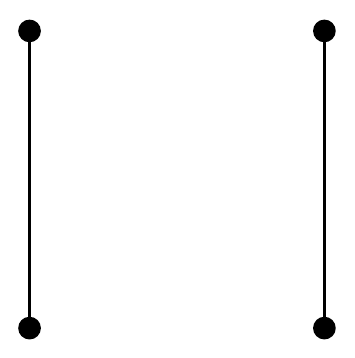} & \includegraphics[width=0.5in,height=0.5in]{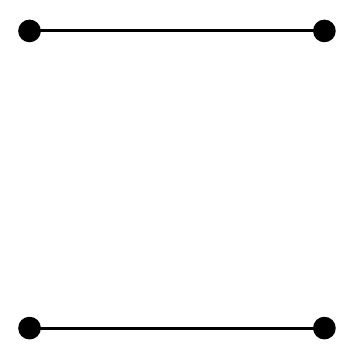}	\\
    		
    	\end{tabular}
    	\caption{Basic figures of different order of a graph $G$.}
    	\label{bfig}
    \end{table} 

	\noindent Let us look at the basic figures containing loops of a general graph $G$ consisting of $p$ vertices, $q$ edges, and $m$ loops (we do not count loops while counting edges). Define $\mathcal{L}$ to be the subset of the vertex set of $G$ consisting of those vertices which have a loop; and let $N(v)$ and $N[v]$ denote the open and closed neighborhoods, respectively, of a vertex $v$ with respect to $G$.
	
	\noindent A basic figure of order $1$ containing loops is just a single vertex with a loop. So such basic figures contribute $-m$ to the coefficient of $x^{p-1}$. Similarly, a basic figure of order $2$ that has loops consists of two loops. Such basic figures contribute ${m\choose 2}$ to the coefficient of $x^{p-2}$. A basic figure of order $3$ containing loops in any graph is either of the forms given in Figure \ref{order3}. The basic figures of the type shown in Figure \ref{order3}(a) contribute $-$${m\choose 3}$ to the coefficient of $x^{p-3}$, whereas the basic figures in Figure \ref{order3}(b) contribute
	\begin{equation*}
		mq - \sum_{v\in \mathcal{L}} |N(v)|.
	\end{equation*}
	This is because, for any vertex $v$ with a loop, we have to pick the edge of the basic figure from the edges which are not incident to $v$. That is equal to subtracting the total number of neighbors of $v$ from the total number of edges in $G$ except for loops.
	\begin{figure}[htb]
		\centering
		\begin{subfigure}{.5\textwidth}
			\centering
			\includegraphics[width=1.2in,height=1in]{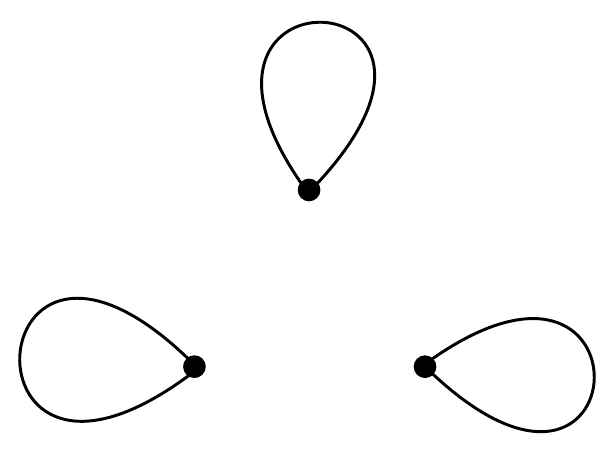}
			\caption{}
		\end{subfigure}%
		\begin{subfigure}{.5\textwidth}
			\centering
			\includegraphics[width=1in,height=0.8in]{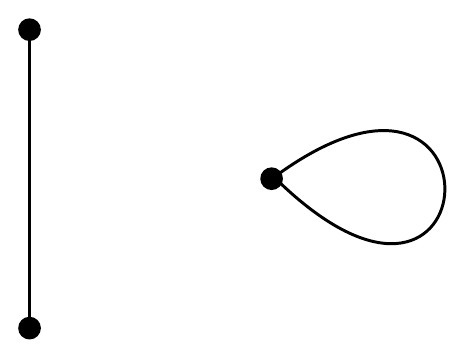}
			\caption{}
		\end{subfigure}
		\caption{Possible basic figures of order $3$ containing loop in any graph.}
		\label{order3}
	\end{figure}	
	
	\noindent For a basic figure of order $4$, we have three possibilities, as shown in Figure \ref{order4}. The basic figures of the type given in Figure \ref{order4}(a) contribute ${m \choose 4}$ to the coefficient of $x^{p-4}$,  whereas the basic figures in Figure \ref{order4}(b) contribute
	\begin{align*}
		&-\bigg[{m\choose 2} q - \sum_{\substack{v_{1},v_{2}\in \mathcal{L}\colon\\ v_{1}v_{2} \in E(G)}} \big(|N(v_{1})| + |N(v_{2})|-1\big) - \sum_{\substack{v_{1},v_{2}\in \mathcal{L}\colon\\ v_{1}v_{2} \notin E(G)}}\big(|N(v_{1})| + |N(v_{2})|\big)\bigg] \\
		&= (m-1) \sum_{v\in \mathcal{L}} |N(v)|-{m\choose 2} q - |E(G[\mathcal{L}])|,
	\end{align*}
	where $\mathcal{L}$ and $N(v)$ are defined above, $E(G)$ is the set of edges of the graph $G$, and $G[\mathcal{L}]$ is the subgraph of $G$ induced by $\mathcal{L}$. Here, first, we have to choose two vertices $v_{1}, v_{2} \in \mathcal{L}$ and then we have to select those edges which are neither incident to $v_{1}$ nor to $v_{2}$. If $v_{1}$ and $v_{2}$ are not adjacent, then we need to subtract the total number of neighbors of $v_{1}$ and $v_{2}$ from the total number of edges in $G$ except for loops. Otherwise, we need to subtract an extra $1$ (as in the second term) as the edge $v_{1}v_{2}$ is counted twice while counting in terms of neighborhood. Next, the basic figures of the type shown in Figure \ref{order4}(c) contribute
	\begin{equation*}
    	\sum_{v\in \mathcal{L}} \big|\mathcal{C}_{3}^{\{v\}}\big|, 
	\end{equation*}
	where $\mathcal{C}_{n}^{\{v\}}$ denotes the set of cycles of length $n$ not containing $v$.
	\begin{figure}[htb]
		\centering
		\begin{subfigure}{.5\textwidth}
			\centering
			\includegraphics[width=1.2in,height=1in]{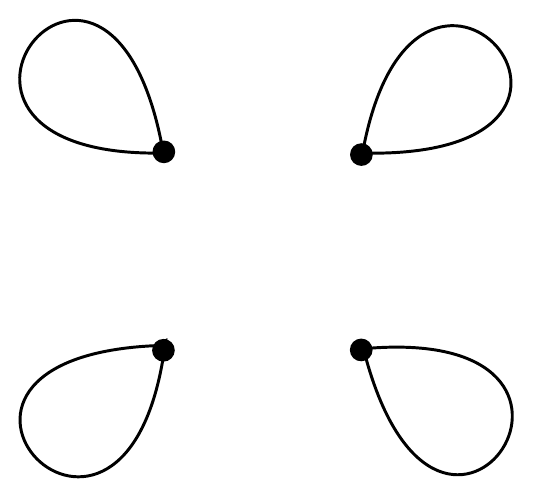}
			\caption{}
		\end{subfigure}%
		\begin{subfigure}{.5\textwidth}
			\centering
			\includegraphics[width=1in,height=1in]{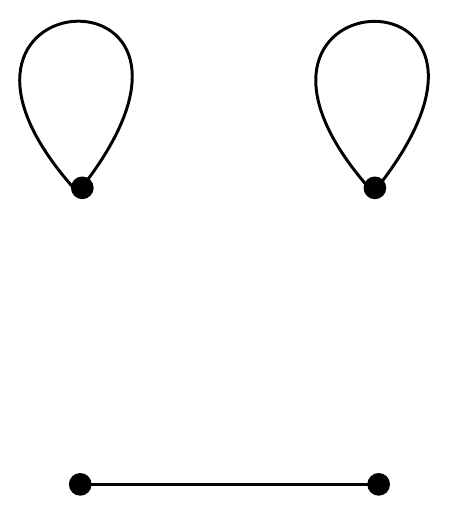}
			\caption{}
		\end{subfigure}\\
	\begin{subfigure}{.5\textwidth}
		\centering
		\includegraphics[width=1in,height=1in]{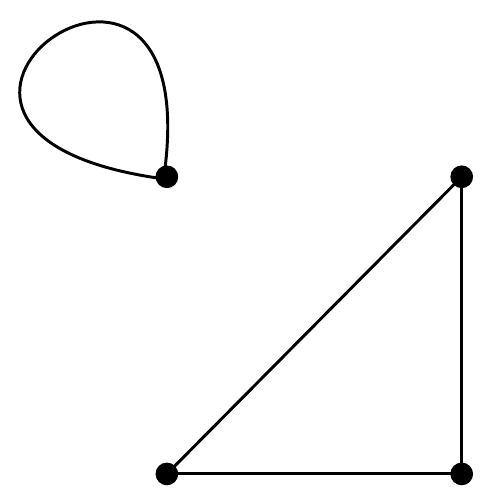}
		\caption{}
	\end{subfigure}
		\caption{Possible basic figures of order $4$ containing loop in any graph.}
		\label{order4}
	\end{figure}

    \noindent This analysis can be further generalized to the basic figures of higher orders. For $k$ such that $1\leq k \leq m$, the basic figures of order $k$ consisting of $k$ loops contribute $(-1)^{k}{m \choose k}$ to the coefficient of $x^{p-k}$. The basic figures of order $k$ consisting of $k-2$ loops and an edge contribute
    \begin{align*}
    	&(-1)^{k-1}\bigg[{m \choose k-2}q - \sum_{\mathcal{L}_{k-2} \subseteq \mathcal{L}} \Big(\sum_{v_{i}\in \mathcal{L}_{k-2}}\big(|N(v_{i})|\big) - |E(G[\mathcal{L}_{k-2}])|\Big)\bigg] \\
	&=(-1)^{k-1}\bigg[{m \choose k-2}q + \binom{m-2}{k-4} |E(G[\mathcal{L}])| - \binom{m-1}{k-3} \sum_{v\in \mathcal{L}} |N(v)|\bigg],
    \end{align*}
	where the first summation runs over all ($k-2$)-element subsets of $\mathcal{L}$. Next, the basic figures of order $k$ consisting of $k-n$ loops and a cycle $C_{n}$ of length $n$ contribute
	\begin{equation*}
		(-1)^{k-n+1}\sum_{\mathcal{L}_{k-n} \subseteq \mathcal{L}} \big|\mathcal{C}_{n}^{\mathcal{L}_{k-n}}\big|,
	\end{equation*}
	where the summation runs over all ($k-n$)-element subsets of $\mathcal{L}$ and $\mathcal{C}_{n}^{\mathcal{L}_{k-n}}$ denotes the set of cycles of length $n$ that do not contains elements of the set $\mathcal{L}_{k-n}$.


	\section{Main Theorems}
	
	 Consider a graph $G$ containing $m$ loops. For $k=m$, the following theorem acts as a bridge between the characteristic polynomials of $G$ and its simple subgraphs.
		
	\begin{theorem}\label{loops+}
		If\/ $l_{1},l_{2}, \dots, l_{k}$ represent loops at vertices $v_{1}, v_{2}, \dots, v_{k}$, respectively, in a graph $G$, then the relationship between the characteristic polynomials of $G$ and $G' = G-\{l_{1},l_{2}, \dots, l_{k}\}$ is given by the following:
		\begin{equation*}
		\phi(G) = \sum_{X\subseteq \mathcal{L}}(-1)^{|X|}\phi(G'-X),
		\end{equation*}
		where $\mathcal{L}=\{v_{1}, v_{2}, \dots, v_{k}\}$.
	\end{theorem}

        \noindent Note that the theorem does not assume $G$ has loops only at the vertices in $\mathcal{L}$.

    \begin{proof}
    	Let $\mathcal{B}\in \textbf{B}(G)$ be a basic figure of $G$ and let $\mathcal{L}(\mathcal{B})$ be the subset of $\mathcal{L}$ consisting of the vertices of $\mathcal{L}$ that are in $\mathcal{B}$. Let $\mathcal{B}'$ be $\mathcal{B}- \mathcal{L}(\mathcal{B})$ viewed as a basic figure of $G'-\mathcal{L}(\mathcal{B})$. Hence, we have $V(\mathcal{B}')= V(\mathcal{B}) -\mathcal{L}(\mathcal{B})$, $c(\mathcal{B}')=c(\mathcal{B})$, and $k(\mathcal{B}')= k(\mathcal{B})- |\mathcal{L}(\mathcal{B})|$. This implies
    	\begin{align*}
    		|V(\mathcal{B}')| &= |V(\mathcal{B}) -\mathcal{L}(\mathcal{B})|\\
    		&= |V(\mathcal{B})| - |V(\mathcal{B}) \cap\mathcal{L}(\mathcal{B})|\\
    		&= |V(\mathcal{B})| - |\mathcal{L}(\mathcal{B})|
    	\end{align*}
    	and
    	\begin{equation*}
    		(-1)^{k(\mathcal{B})}2^{c(\mathcal{B})}= (-1)^{|\mathcal{L}(\mathcal{B})|}(-1)^{k(\mathcal{B}')}2^{c(\mathcal{B}')}.
    	\end{equation*}
    	The proof of the theorem is now a computation,
    	\begin{align*}
    		\phi(G)&= \sum_{\mathcal{B} \in \textbf{B}(G)} (-1)^{k(\mathcal{B})} 2^{c(\mathcal{B})}x^{p-|V(\mathcal{B})|}\\
    		&= \sum_{X\subseteq \mathcal{L}} \sum_{\substack{\mathcal{B}\in \textbf{B}(G)\colon\\ \mathcal{L}(\mathcal{B})=X}} (-1)^{|\mathcal{L}(\mathcal{B})|}(-1)^{k(\mathcal{B}')}2^{c(\mathcal{B}')} x^{p-|V(\mathcal{B}')|-|\mathcal{L}(\mathcal{B})|}\\
    		&= \sum_{X\subseteq \mathcal{L}}(-1)^{|X|} \sum_{\mathcal{B}'\in \textbf{B}(G'- X)} (-1)^{k(\mathcal{B}')}2^{c(\mathcal{B}')} x^{p-|V(\mathcal{B}')|-|X|}\\
    		&= \sum_{X\subseteq \mathcal{L}}(-1)^{|X|} \phi(G'-X).\qedhere
    	\end{align*}
    \end{proof}

        \noindent The result in Theorem \ref{loops+} is useful when the characteristic polynomial of a graph without loops is known, and we wish to compute the characteristic polynomial of the same graph containing loops at some or all vertices.  

    \begin{coro}\label{G-loop}
    	If\/ $l$ represents a loop at vertex $v$ in a graph $G$, then the characteristic polynomial of $G$ satisfies
    	\begin{equation*}
    	\phi(G) = \phi(G-l) - \phi(G-v).
    	\end{equation*}
   \end{coro}

    \begin{coro}
    	In the characteristic polynomial of $G$, the coefficient $a_{i}(G)$ of $x^{p-i}$ is
    	\begin{equation*}
    		a_{i}(G) = \sum_{X \subseteq \mathcal{L}\colon |X|\leq i} (-1)^{|X|} a_{i-|X|}(G'-X).
    	\end{equation*}
    \end{coro}
    \begin{proof}
    	From Theorem \ref{loops+}, we have
    	\begin{align*}
    		\phi(G) &= \sum_{X\subseteq \mathcal{L}}(-1)^{|X|}\phi(G'-X)\\
    		&= \sum_{X \subseteq \mathcal{L}} (-1)^{|X|} \sum_{j=0}^{p-|X|} a_{j}(G'-X) x^{p-|X|-j}
    		\intertext{and now replacing the sum over $j$ by a sum over $i$, where $i = j+|X|$,}
    		&= \sum_{X \subseteq \mathcal{L}} \sum_{i=|X|}^{p} x^{p-i} (-1)^{|X|} a_{i-|X|}(G'-X)
    		\intertext{and by reversing the order of summation, we get}
    		&= \sum_{i=0}^{p} x^{p-i} \sum_{\substack{X \subseteq \mathcal{L}\colon\\|X|\leq i}} (-1)^{|X|} a_{i-|X|}(G'-X).
    	\end{align*}
    	The corollary follows by extracting the coefficient of $x^{p-i}$.
    \end{proof}

    \noindent In Theorem \ref{loops+}, a basic figure may appear implicitly several times on the right-hand side.  Here are formulas that express quantities in terms of sums explicitly over basic figures.
    
    \begin{theorem}
    	The characteristic polynomial of $G$ satisfies
    	\begin{equation*}
    		\phi(G) =  \sum_{\mathcal{B}' \in \textbf{B}(G')} (-1)^{k(\mathcal{B}')} 2^{c(\mathcal{B}')} x^{p-|V(\mathcal{B}') \cup \mathcal{L}|} ({x-1})^{|\mathcal{L} - V(\mathcal{B}')|}
    	\end{equation*}
    	and the coefficient of $x^{p-i}$ is
    	\begin{equation*}
    		a_{i}(G) = \sum_{\substack{\mathcal{B}' \in \textbf{B}(G')\colon\\|V(\mathcal{B}')| \leq i \leq |V(\mathcal{B}')\cup \mathcal{L}|}}  (-1)^{k(\mathcal{B}') - |V(\mathcal{B}')| + i} 2^{c(\mathcal{B}')} \binom{|\mathcal{L} - V(\mathcal{B}')|}{|V(\mathcal{B}') \cup \mathcal{L}| - i}.
    	\end{equation*}
    	In particular,
    	\begin{equation*}
    		\det A(G) = (-1)^p a_{p}(G) = \sum_{\substack{\mathcal{B}' \in \textbf{B}(G')\colon\\V(\mathcal{B}')\supseteq V(G)-\mathcal{L}}}  (-1)^{k(\mathcal{B}') - |V(\mathcal{B}')|} 2^{c(\mathcal{B}')}.
    	\end{equation*}
    \end{theorem}
    \begin{proof}
    	The formulae are obtained by expanding the result in Theorem \ref{loops+}. We have
    	\begin{align*}
    		\phi(G) &= \sum_{X\subseteq \mathcal{L}}(-1)^{|X|}\phi(G'-X)\\
    		&=\sum_{X\subseteq \mathcal{L}}(-1)^{|X|} \sum_{\mathcal{B}'\in \textbf{B}(G'- X)} (-1)^{k(\mathcal{B}')}2^{c(\mathcal{B}')} x^{p-|V(\mathcal{B}')|-|X|}
    		\intertext{and by interchanging the order of summation, we get}
    		&= \sum_{\mathcal{B}' \in \textbf{B}(G')} (-1)^{k(\mathcal{B}')} 2^{c(\mathcal{B}')} x^{p-|V(\mathcal{B}')|} \sum_{X \subseteq \mathcal{L} - V(\mathcal{B}')} \Big(-\frac{1}{x}\Big)^{|X|} \\
    		&= \sum_{\mathcal{B}' \in \textbf{B}(G')} (-1)^{k(\mathcal{B}')} 2^{c(\mathcal{B}')} x^{p-|V(\mathcal{B}')|} \Big(\frac{x-1}{x}\Big)^{|\mathcal{L} - V(\mathcal{B}')|} \\
    		&= \sum_{\mathcal{B}' \in \textbf{B}(G')} (-1)^{k(\mathcal{B}')} 2^{c(\mathcal{B}')} x^{p-|V(\mathcal{B}')|-|\mathcal{L} - V(\mathcal{B}')|} ({x-1})^{|\mathcal{L} - V(\mathcal{B}')|}
    		\intertext{and since $|V(\mathcal{B}')|+|\mathcal{L} - V(\mathcal{B}')| = |V(\mathcal{B}') \cup \mathcal{L}|$, we get}
    		&= \sum_{\mathcal{B}' \in \textbf{B}(G')} (-1)^{k(\mathcal{B}')} 2^{c(\mathcal{B}')} x^{p-|V(\mathcal{B}') \cup \mathcal{L}|} ({x-1})^{|\mathcal{L} - V(\mathcal{B}')|}\\
    		&= \sum_{\mathcal{B}' \in \textbf{B}(G')} (-1)^{k(\mathcal{B}')} 2^{c(\mathcal{B}')} x^{p-|V(\mathcal{B}') \cup \mathcal{L}|} \sum_{j=0}^{|\mathcal{L} - V(\mathcal{B}')|} x^{j} (-1)^{|\mathcal{L} - V(\mathcal{B}')|-j} \binom{|\mathcal{L} - V(\mathcal{B}')|}{j} \\
    		&= \sum_{\mathcal{B}' \in \textbf{B}(G')} \sum_{j=0}^{|\mathcal{L} - V(\mathcal{B}')|} (-1)^{k(\mathcal{B}')+|\mathcal{L} - V(\mathcal{B}')|-j} 2^{c(\mathcal{B}')} x^{p-|V(\mathcal{B}') \cup \mathcal{L}| + j} \binom{|\mathcal{L} - V(\mathcal{B}')|}{j}
    		\intertext{and by replacing the sum over $j$ by a sum over $i$, where $i = |V(\mathcal{B}') \cup \mathcal{L}| - j$, we get}
    		&= \sum_{\mathcal{B}' \in \textbf{B}(G')} \sum_{i=|V(\mathcal{B}')|}^{|V(\mathcal{B}')\cup \mathcal{L}|} (-1)^{k(\mathcal{B}')+|\mathcal{L} - V(\mathcal{B}')| - |V(\mathcal{B}') \cup \mathcal{L}| + i} 2^{c(\mathcal{B}')} x^{p-i} \binom{|\mathcal{L} - V(\mathcal{B}')|}{|V(\mathcal{B}') \cup \mathcal{L}| - i} \\
    		&= \sum_{\mathcal{B}' \in \textbf{B}(G')} \sum_{i=|V(\mathcal{B}')|}^{|V(\mathcal{B}')\cup \mathcal{L}|} (-1)^{k(\mathcal{B}') - |V(\mathcal{B}')| + i} 2^{c(\mathcal{B}')} x^{p-i} \binom{|\mathcal{L} - V(\mathcal{B}')|}{|V(\mathcal{B}') \cup \mathcal{L}| - i}.
    	\end{align*}
    	Finally, by reversing the order of summation again, we have $i$ such that $0 \leq i \leq p$ and the basic figure $\mathcal{B}'$ must satisfy $|V(\mathcal{B}')| \leq i \leq |V(\mathcal{B}')\cup \mathcal{L}|.$  It implies
    	\begin{align*}
    	\phi(G) &= \sum_{i=0}^{p} x^{p-i}  \sum_{\substack{\mathcal{B}' \in \textbf{B}(G')\colon\\|V(\mathcal{B}')| \leq i \leq |V(\mathcal{B}')\cup \mathcal{L}|}}  (-1)^{k(\mathcal{B}') - |V(\mathcal{B}')| + i} 2^{c(\mathcal{B}')} \binom{|\mathcal{L} - V(\mathcal{B}')|}{|V(\mathcal{B}') \cup \mathcal{L}| - i},
    	\end{align*}
    	where the inner sum is the coefficient of $x^{p-i}$.
    	
    	\noindent For $i=p$, the inner summation simplifies to 
    	\begin{equation*}
    		a_{p} = (-1)^{p}\sum_{\substack{\mathcal{B}' \in \textbf{B}(G')\\V(\mathcal{B}')\supseteq V(G)-\mathcal{L}}}  (-1)^{k(\mathcal{B}') - |V(\mathcal{B}')|} 2^{c(\mathcal{B}')}
    	\end{equation*}
    	and the result follows.
    \end{proof}
    
    \noindent The following theorem helps in computing the characteristic polynomial of a graph containing loops. For a vertex $v$ and a set $X$ of vertices, let $\mathcal{C}(v)^{X}$ denote the set of cycles containing $v$ that do not contain elements of $X$.
    
	\begin{theorem}\label{vertices}
	Let $v_{1}, v_{2}, \dots, v_{k}$ be some of the vertices of a graph $G$. The relationship between the characteristic polynomials of $G$ and $G-\{v_{1},v_{2}, \dots, v_{k}\}$ is given by the following:
	\begin{multline*}
      	\phi(G) = x^{k}\phi(G-X_{k}) - \sum_{i=0}^{k-1}	x^{i}\Big[\sum_{v \in N[v_{i+1}]-X_{i}} \phi(G-X_{i+1}-v)\\
	    + 2 \sum_{C\in \mathcal{C}(v_{i+1})^{X_{i}}} \phi\big(G-X_{i}-V(C)\big)\Big],
	\end{multline*}
	where $X_{0}=\{\}$ and for $i\geq 1$, $X_{i}=\{v_{1}, v_{2}, \dots, v_{i}\}$.
    \end{theorem}
    \begin{proof}
    	Let $\mathcal{B}\in \textbf{B}(G)$ be a basic figure of $G$. We give a one-to-one correspondence between the basic figures of $G$ and those contributing to one of the terms on the right. Consider the following cases:
    	
    	\noindent \emph{Case 1. If $v_{1} \in \mathcal{B}$.}
    	
    	\noindent (a) If $v_{1}\in K_{2} \subseteq \mathcal{B}$, let $\mathcal{B}'$ be $\mathcal{B}-V(K_{2})$ viewed as a basic figure of $G-V(K_{2})$. We have $k(\mathcal{B}')=k(\mathcal{B})-1$. It implies $(-1)^{k(\mathcal{B}')} 2^{c(\mathcal{B}')} = -(-1)^{k(\mathcal{B})} 2^{c(\mathcal{B})}$. Hence, $\mathcal{B}'$ contributes $-(-1)^{k(\mathcal{B})} 2^{c(\mathcal{B})}$ to $\phi(G-v-v_{1})$, where $v\in K_{2}$. In turn, $\mathcal{B}'$ contributes the same amount in $-\phi(G-v-v_{1})$ as does $\mathcal{B}$ in $\phi(G)$. This holds for all $v\in N(v_{1})$.
    	
    	\noindent (b) If $v_{1}\in C \subseteq \mathcal{B}$, let $\mathcal{B}'$ be $\mathcal{B}- V(C)$ viewed as a basic figure of $G-V(C)$. Here, $k(\mathcal{B}')=k(\mathcal{B})-1$ and $c(\mathcal{B}')=c(\mathcal{B})-1$. It implies $(-1)^{k(\mathcal{B}')} 2^{c(\mathcal{B}')} = (-1)^{k(\mathcal{B})-1} 2^{c(\mathcal{B})-1}=\frac{-1}{2}(-1)^{k(\mathcal{B})} 2^{c(\mathcal{B})}$. Hence, $\mathcal{B}'$ contributes $\frac{-1}{2}(-1)^{k(\mathcal{B})} 2^{c(\mathcal{B})}$ to $\phi(G-V(C))$. In turn, $\mathcal{B}'$ contributes the same amount in $-2\phi(G-V(C))$ as does $\mathcal{B}$ in $\phi(G)$. This holds for every cycle containing $v_{1}$.
    	
    	\noindent (c) If $v_{1}\in \mathcal{B}$ with a loop at it, let $\mathcal{B}$ be $\mathcal{B}-\{v_{1}\}$ viewed as a basic figure of $G-v_{1}$. We have $k(\mathcal{B}')=k(\mathcal{B})-1$. Analogously to the first subcase, $\mathcal{B}'$ contributes the same amount in $-\phi(G-v_{1})$ as does $\mathcal{B}$ in $\phi(G)$.
    	
    	\noindent Combining the above subcases, $\mathcal{B}'$ contributes in $\phi(G-v-v_{1})$ and $\phi(G-V(C))$ for all $v\in N[v_{1}]$ and for every cycle containing $v_{1}$.
    	
    	\noindent \emph{Case 2. If $v_{1} \notin \mathcal{B}$ and $v_{2} \in \mathcal{B}$.}
    	
    	\noindent (a) If $v_{2}\in K_{2} \subseteq \mathcal{B}$, let $\mathcal{B}'$ be $\mathcal{B}- V(K_{2})$ viewed as a basic figure of $G-v_{1}-V(K_{2})$. Here, $\mathcal{B}'$ contributes $-(-1)^{k(\mathcal{B})} 2^{c(\mathcal{B})}$ in $x\phi(G-v-v_{1}-v_{2})$, where $v\in K_{2}$. If $X_{2}= \{v_{1}, v_{2}\}$, then $\mathcal{B}'$ contributes in $-x\phi(G-X_{2}-v)$ for all $v \in N(v_{2})-v_{1}$.
    		
        \noindent (b) If $v_{2}\in C \subseteq \mathcal{B}$, let $\mathcal{B}'$ be $\mathcal{B}- V(C)$ viewed as a basic figure of $G-v_{1}-V(C)$. Here, $\mathcal{B}'$ contributes $\frac{-1}{2}(-1)^{k(\mathcal{B})} 2^{c(\mathcal{B})}$ to $x\phi(G-v_{1}-V(C))$. In turn, $\mathcal{B}'$ contributes in $-2x\phi(G-X_{1}-V(C))$, where $X_{1}=\{v_{1}\}$.
        
        \noindent (c) If $v_{2}\in \mathcal{B}$ with a loop at it, let $\mathcal{B}$ be $\mathcal{B}-\{v_{2}\}$ viewed as a basic figure of $G-\{v_{1},v_{2}\}$. In this case, $\mathcal{B}'$ contributes in $-x\phi(G-X_{2})$.
        
        \noindent Here, $\mathcal{B}'$ contributes in $\phi(G-X_{2}-v)$ and $\phi(G-X_{1}-V(C))$ for all $v\in N[v_{2}]-X_{1}$ and for every cycle $C\in\mathcal{C}(v_{2})^{X_{1}}$.
        
        \noindent Proceeding likewise, the next case would be where the vertices $v_{1}$ and $v_{2}$ do not belong to $\mathcal{B}$ but $v_{3}\in \mathcal{B}$; a similar analysis can be done. Define $X_{i}=\{v_{1}, v_{2}, \dots, v_{i}\}$ for all $i$ such that $3\leq i\leq k$. The process continues till none of the vertices $v_{1}, v_{2}, \dots, v_{k}$ is present in $\mathcal{B}$. Let $\mathcal{B}'$ be the same basic figure viewed as a subgraph of $G-X_{k}$. Here, $\mathcal{B}'$ contributes the same amount in $x^{k}\phi(G- X_{k})$ as does $\mathcal{B}$ in $\phi(G)$. The result follows by combining all cases.       
    \end{proof}
    
    \noindent If we want to remove one vertex, the following corollary coincides with a result in \cite{sch}.
    
    \begin{coro}
    	Let $v$ be a vertex of a graph $G$ and let $\mathcal{C}(v)$ denote the set of cycles containing $v$. The characteristic polynomial $\phi(G)$ satisfies
    	\begin{equation*}
    		\phi(G) = x\phi(G-X_{1}) - \sum_{u \in N[v]} \phi(G-X_{1}-u) - 2\sum_{C\in\mathcal{C}(v)} \phi(G-V(C)), 
    	\end{equation*}
    	where $X_{1}= \{v_{1}\}$.
    \end{coro}

    \noindent The following theorem states another relationship between the characteristic polynomials of graphs $G$ and $G'$. This time the theorem is useful when the characteristic polynomial of $G$ is known, and we need to compute the characteristic polynomial of $G'$. In the theorem, to obtain the characteristic polynomial of $G'$, we keep on removing loops from $G$, one by one, until we get $G'$.

    \begin{theorem}\label{loops-}
    	If\/ $l_{1},l_{2}, \dots, l_{k}$ represent loops at vertices $v_{1}, v_{2}, \dots, v_{k}$, respectively, in a graph $G$, then the relationship between the characteristic polynomials of $G$ and $G'=G-\{l_{1},l_{2}, \dots, l_{k}\}$ is given by the following:
    	\begin{equation*}
    		\phi(G')= \phi(G) + \sum_{i=1}^{k}\phi(G - \{l_{1}, l_{2}, \dots, l_{i-1}\}- v_{i}).
    	\end{equation*}
    \end{theorem}

    \noindent Note that we do not assume $G'$ has loops only at the vertices $v_{1}, v_{2}, \dots, v_{k}$.

    \begin{proof}
    	We prove the result using the principle of mathematical induction on the number of loops. Consider the graph $G$ containing loops $l_{1},l_{2}, \dots, l_{k}$ at vertices $v_{1}, v_{2}, \dots, v_{k}$, respectively. For the case $k=1$, on the left-hand side, we have $\phi(G')$, where $G'=G-l_{1}$; and expanding the summation term, we get as right-hand side $\phi(G)+ \phi(G-v_{1})$. We need to prove that $\phi(G-l_{1})$ is equal to $\phi(G)+\phi(G-v_{1})$, which is true by Corollary \ref{G-loop}. This case describes the relationship between the characteristic polynomials of a graph $G$ and a graph obtained upon removing one loop from $G$.
    	Let us assume the result is true for $k=j-1$; we have
    	\begin{equation*}
    		\phi(G'_{j-1}) = \phi(G) + \sum_{i=1}^{j-1} \phi(G-\{l_{1}, l_{2}, \dots, l_{i-1}\}- v_{i}),
    	\end{equation*}
    	where $G'_{j-1} = G-\{l_{1}, l_{2}, \dots, l_{j-1}\}$. We prove the result for $k=j$. Consider the graph $G'_{j}$, which is $G-\{l_{1}, l_{2}, \dots, l_{j}\}$. It can be written in terms of $G'_{j-1}$ as follows:
    	\begin{equation*}
    		G'_{j}= G-\{l_{1}, l_{2}, \dots, l_{j-1}\}-l_{j}= G'_{j-1}- l_{j}.
    	\end{equation*}
    	Hence, we can obtain the graph $G'_{j}$ from the graph $G'_{j-1}$ by removing the loop $l_{j}$. So we can apply the case where one loop is being removed. We have 
    	\begin{align*}
    		\phi(G'_{j}) &= \phi(G'_{j-1}-l_{j})\\
    		&= \phi(G'_{j-1})+\phi(G'_{j-1} -v_{j}).
    	    \intertext{Using the case $k=j-1$, we get}
    		\phi(G'_{j}) &= \phi(G) + \sum_{i=1}^{j-1} \phi(G-\{l_{1}, l_{2}, \dots, l_{i-1}\}- v_{i})+\phi(G'_{j-1} -v_{j})\\
    		&= \phi(G) + \sum_{i=1}^{j-1} \phi(G-\{l_{1}, l_{2}, \dots, l_{i-1}\}- v_{i})+\phi(G-\{l_{1}, l_{2}, \dots, l_{j-1}\} -v_{j})\\
    		&= \phi(G) + \sum_{i=1}^{j} \phi(G-\{l_{1}, l_{2}, \dots, l_{i-1}\}- v_{i}).
    	\end{align*}
    	Hence, the result is true for $k=j$. By the principle of mathematical induction, the result is true for any $k$.
    \end{proof}

    \begin{example}
    	We wish to find the characteristic polynomial of the graph $G$ shown in Figure \ref{Exam}.
    	\begin{figure}[H]
    		\centering
    		\includegraphics[width=1.5in,height=1.2in]{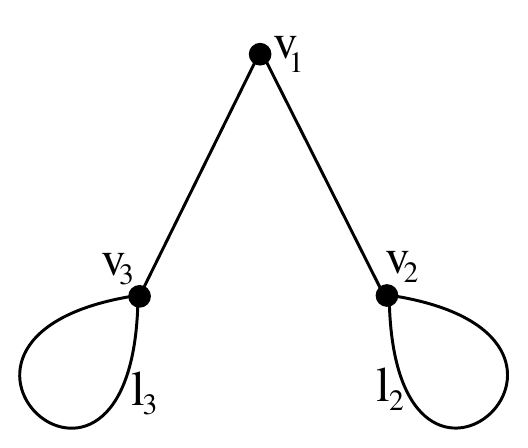}
    		\caption{A graph $G$.}
    		\label{Exam}
    	\end{figure}
        \noindent Using Theorem \ref{loops+}, the characteristic polynomial of $G$ satisfies
        \begin{align*}
        \phi(G)&= \phi(G') - \phi(G'-v_{2}) -\phi(G'-v_{3}) + \phi(G'- \{v_{2}, v_{3}\})\\
        &= \phi(P_{3}) - \phi(P_{2}) - \phi(P_{2}) + \phi(P_{1})\\
        &= \phi(P_{3}) - 2\phi(P_{2}) + \phi(P_{1})\\
        &= x^{3} - 2x - 2(x^{2}-1) + x\\
        &= x^{3} - 2x^{2} - x + 2,
        \end{align*}
        where $P_{n}$ denotes the path graph on $n$ vertices.
    \end{example}

    \begin{example}
    	We wish to find the characteristic polynomial of the graph $G'$ obtained upon removing loops from the graph $G$ shown in Figure \ref{Exam}. Using Theorem \ref{loops-}, the characteristic polynomial of $G'$ satisfies
        \begin{align*}
        \phi(G')&= \phi(G) + \phi(G-v_{2}) + \phi(G-l_{2}-v_{3})\\
        &= x^{3} - 2x^{2} - x + 2 + x^{2} - x - 1 + x^{2} - 1\\
        &= x^{3} - 2x.
        \end{align*}
    \end{example}


\section{Unitary Addition Cayley Graph}

    For a positive integer $n$, the unitary addition Cayley graph $G_{n}$ is a simple graph whose vertex set is $\mathbb{Z}_{n}= \{0, 1, \dots, n-1\}$, the ring of integers modulo $n$ and in which two vertices $v_{1}$ and $v_{2}$ are adjacent if and only if  $v_{1}+v_{2}\in U(n)$, where $U(n)$ denotes the set of units of the ring $\mathbb{Z}_{n}$. By relabeling the vertices as $v_{i}=i-1$, the adjacency matrix associated with $G_{n}$ is an $n\times n$ matrix $A(G_{n})=[a_{ij}]$ whose rows (and columns) correspond to integers $0,1,\dots, n-1$ and which is such that $a_{ii}=0$ and for $i\neq j$,
    \begin{equation*}
    	a_{ij}= \begin{cases}
    	1 & \text{if } (i-1)+(j-1) \in U(n), \\
    	0 & \text{if } (i-1)+(j-1) \notin U(n).
    	\end{cases}
    \end{equation*}
    Consider an anti-circulant matrix $A_{n}$ of order $n$ with first row $a_{0}, a_{1}, \dots, a_{n-1}$:
    \begin{equation*}\label{eq 15}
    A_{n}= \begin{bmatrix}
    a_{0} & a_{1} & a_{2} & \cdots & a_{n-1} \\
    a_{1} & a_{2} & a_{3} & \cdots & a_{0} \\
    \vdots & \vdots & \vdots & \ddots & \vdots \\
    a_{n-1} & a_{0} & a_{1} & \cdots & a_{n-2}
    \end{bmatrix}
    \end{equation*}
    Note that, the $(i,j)$ entry of $A_{n}$ is $a_{i-1+j-1\pmod {n}}$. In \cite{amot}, the authors gave an explicit formula for the eigenvalues of an anti-circulant matrix using the entries $a_{0}, a_{1}, \dots, a_{n-1}$. Define a polynomial $p_{n}(x)$ as:
    \begin{equation*}
    	p_{n}(x)= \sum_{j=0}^{n-1} a_{j}x^{j}.
    \end{equation*}
    For $r=0,1,\dots,\lfloor\frac{n}{2}\rfloor$, define $\lambda_{r}= p_{n}(w^{r})$, where $w=e^{2\pi\iota/n}$, a complex primitive $n$-th root of unity.
    The eigenvalues of $A_{n}$ are as follows (cf. \cite[Theorem 3.6]{amot}):
    \begin{equation*}
    \begin{cases}
    	\lambda_{0}, \lambda_{n/2}, \pm|\lambda_{r}| & \text{if $n$ is even},\\
    	\lambda_{0}, \pm |\lambda_{r}| & \text{if $n$ is odd},
    \end{cases}
    \end{equation*}
    where $r=1,2, \dots,\lfloor\frac{n-1}{2}\rfloor$ and $\lambda_{n/2}=\sum_{j=0}^{n-1}(-1)^{j}a_{j}$.
     
    \noindent As a special case, take
    \begin{equation*}
    a_{j}= \begin{cases}
    1 & \text{if } \gcd(j,n)=1, \\
    0 & \text{if } \gcd(j,n) > 1,
    \end{cases}
    \end{equation*}
    where $\gcd(a,b)$ denotes the greatest common divisor of integers $a$ and $b$. For $r=0,1, \dots,\lfloor\frac{n}{2}\rfloor$, the expression for eigenvalue $\lambda_{r}$ becomes
    \begin{equation*}
    \lambda_{r}= \sum_{\substack{0\leq j\leq n-1\\\gcd(j,n)=1}} w^{rj}.
    \end{equation*}
    In particular, $\lambda_{0}= \phi(n)$, where here $\phi$ denotes Euler's totient function. Also, for even values of $n$, since the set of units contains only odd integers, $\lambda_{n/2}$ becomes $-\phi(n)$.
    
    \noindent Let $X(A_{n})$ be the graph associated with the matrix $A_{n}$.
    
    \begin{theorem}\label{antichar}
    	The characteristic polynomial of $X(A_{n})$ is given by:
    	\begin{equation*}
    	\phi(X(A_{n}))=\begin{cases}
    	\big(x^{2}-\phi(n)^{2}\big) \prod_{r=1}^{\lfloor\frac{n-1}{2}\rfloor}(x^{2}-|\lambda_{r}|^{2}) & \text{if $n$ is even},\\[6pt]
    	\big(x-\phi(n)\big)\prod_{r=1}^{\lfloor\frac{n-1}{2}\rfloor}(x^{2}-|\lambda_{r}|^{2}) & \text{if $n$ is odd}.
    	\end{cases}
    	\end{equation*}
    \end{theorem}
    
    \noindent For $i=1,2,\dots, n$, consider the $(i,i)$ diagonal entry $a_{i-1+i-1\pmod {n}}$ of $A_{n}$. This entry is $1$ if and only if $\gcd(2i-2,n)=1$. Therefore, the graph $X({A_{n}})$ contains loops. The following theorem connects the unitary addition Cayley graph $G_{n}$ with the graph $X(A_{n})$.
    
    \begin{theorem}\label{iso}
    	Let $n$ be a positive integer. 
	
	For even values of $n$, the unitary addition Cayley graph $G_{n}$ is isomorphic to $X(A_{n})$. 
	
	For odd values of $n$, $G_{n}$ is isomorphic to the graph obtained from $X(A_{n})$ by removing all loops. The loops in $X(A_n)$ are at the vertices $v_{i}$ for $i-1 \in U(n)$.
    \end{theorem}
    \begin{proof}
    	We prove the result by comparing the adjacency matrices for both graphs. For $i\neq j$, irrespective of $n$, 
    	\begin{align*}
    		[a_{ij}]_{A(G_{n})} = 1 &\iff i-1+j-1 \in U(n)\\
    		&\iff \gcd (i-1+j-1, n)=1\\
    		&\iff a_{i-1+j-1\pmod {n}}=1\\
    		&\iff [a_{ij}]_{A_{n}}=1.
    	\end{align*}
    	For $i=j$, $a_{ij}=0$ in $A(G_{n})$. However, $a_{ij}=1$ in $A_{n}$ if and only if $\gcd(2i-2, n)=1$. For even values of $n$, $\gcd(2i-2, n)\neq 1$. Hence, $a_{ij}=0$ in $A_{n}$. Therefore, the unitary addition Cayley graph $G_{n}$ is isomorphic to $X(A_{n})$. For odd values of $n$, $\gcd(2i-2, n)=1$ if and only if $i-1\in U(n)$. Since these entries represent loops, therefore $G_{n}$ is isomorphic to the graph obtained from $X(A_{n})$ by removing all loops.
    \end{proof}
    
    \noindent Finding the exact eigenvalues of unitary addition Cayley graphs is still an open problem. The authors in \cite{palc} tried to obtain bounds on these values. With the help of the above result and the results in previous section, the characteristic polynomial of an unitary addition Cayley graph can be computed, and hence, one can find the eigenvalues as zeros of this polynomial.

    \begin{theorem}
    	For an even positive integer $n$, the characteristic polynomial of $G_{n}$ is given by the following:
    	\begin{equation*}
    	\phi(G_{n})= \big(x^{2}-\phi(n)^{2}\big)\prod_{r=1}^{\lfloor\frac{n-1}{2}\rfloor}(x^{2}-|\lambda_{r}^{2}).
    	\end{equation*}
    	In particular, the eigenvalues of $G_{n}$ are $\pm\phi(n)$ and $\pm|\lambda_{r}|$.
    \end{theorem}

    \begin{theorem}
    	For an odd positive integer $n$, the characteristic polynomial of $G_{n}$ satisfies
    	\begin{equation*}
    	\phi(G_{n})= \big(x-\phi(n)\big)\prod_{r=1}^{\lfloor\frac{n-1}{2}\rfloor}(x^{2}-\lambda_{r}^{2}) + \sum_{i\in U(n)} \phi\big(X(A_{n})-\sum_{\substack{j\in U(n)\colon\\ j<i}} l_{j}-i\big),
    	\end{equation*}
    	where $l_{j}$ represents a loop at vertex $j$.
    \end{theorem}
    \begin{proof}
    	From Theorem \ref{iso}, the unitary addition Cayley graph $G_{n}$ is isomorphic to the graph obtained from $X(A_{n})$ by removing the loop at vertex $i$ for all $i$ such that $i\in U(n)$. Now the result follows from Theorem \ref{loops-} and Theorem \ref{antichar} by taking the units as vertices in increasing order.
    \end{proof}


    \section{Acknowledgments}

    The first author's work is supported by the Research Grant from DST [MTR/2018/000607] under Mathematical Research Impact Centric Support(MATRICS) for a period of 3-years (2019--2022). The second author is thankful to the University Grant Commission (UGC) for providing the research grant vide sanctioned letter number 1054(CSIR-UGC NET JUNE 2017).

\bibliographystyle{siam}
\bibliography{ref}

\end{document}